\documentclass[a4paper, pdftex]{amsart}

\usepackage[utf8]{inputenc}
\usepackage[T1]{fontenc}
\usepackage[pdftex]{graphicx}
\usepackage{amsmath}
\usepackage{amsfonts}
\usepackage{amsthm}
\usepackage{amssymb}
\usepackage{setspace}	
\usepackage[format=hang]{caption}
\usepackage{thmtools}
\usepackage{thm-restate}
\usepackage{mathtools}
\usepackage{verbatim}
\usepackage[colorlinks,citecolor=blue]{hyperref}

\DeclareMathOperator{\Aut}{Aut}

\begin{document}
\newtheorem*{thm*}{Theorem}
\newtheorem{theorem}{Theorem}[section]
\newtheorem{corollary}[theorem]{Corollary}
\newtheorem{lemma}[theorem]{Lemma}
\newtheorem{fact}[theorem]{Fact}
\newtheorem*{fact*}{Fact}
\newtheorem{proposition}[theorem]{Proposition}
\newtheorem{thmletter}{Theorem}
\renewcommand*{\thethmletter}{\Alph{thmletter}}
\theoremstyle{definition}
\newtheorem{question}[theorem]{Question}
\newtheorem{definition}[theorem]{Definition}
\theoremstyle{remark}
\newtheorem{remark}[theorem]{Remark}
\newtheorem{example}[theorem]{Example}

\newcommand{\set}[1][ ]{\ensuremath{ \lbrace #1 \rbrace}}
\newcommand{\bsl}{\ensuremath{\setminus}}
\newcommand{\grep}[2]{\ensuremath{\left\langle #1 | #2\right\rangle}}
\renewcommand{\ll}{\left\langle}
\newcommand{\rr}{\right\rangle}

\newcommand{\TheDim}{m}
\newcommand{\TheCodim}{k}
\newcommand{\TheRank}{n}
\newcommand{\GL}[2][\TheRank]{\ensuremath{\operatorname{GL_{#1}}(#2)}}
\newcommand{\Stab}{\operatorname{Stab}}
\newcommand{\Opp}{\ensuremath{\operatorname{Opp}(\Delta)}}
\newcommand{\op}{\operatorname{op}}
\newcommand{\mcU}{\ensuremath{\mathcal{U}}}
\newcommand{\mcP}{\ensuremath{\mathcal{P}}}
\newcommand{\mcH}{\ensuremath{\mathcal{H}}}
\newcommand{\FC}{\ensuremath{\operatorname{FC}_{\TheRank}}}
\newcommand{\CC}[2]{\ensuremath{\operatorname{CC}( #1 ,\, #2})}
\newcommand{\real}[1]{\ensuremath{\left\lVert #1\right\rVert}}

\title{Higher generating subgroups and Cohen--Macaulay complexes}
\author{Benjamin Br\"uck}
\begin{abstract}
We show how to find higher generating families of subgroups, in the sense of Abels and Holz, for groups acting on Cohen--Macaulay complexes. We apply this to groups with a BN-pair to prove higher generation by parabolic and Levi subgroups and describe higher generating families of parabolic subgroups in $\Aut(F_n)$.
\end{abstract}

\maketitle

\section{Introduction}
For a group $G$ with a family of subgroups $\mcH$, Abels and Holz in \cite{AH:Highergenerationsubgroups} defined the notion of ``higher generation'' of $G$ by $\mcH$. Whether or not $G$ is highly generated by $\mcH$ depends on the connectivity properties of the nerve of the covering of $G$ given by the set of cosets $\set[gH \mid g \in G,\, H \in \mcH]$ (for precise definitions see Section \ref{sec higher generation}).

Abels and Holz connected higher generating families to finiteness properties of groups, recent work in this direction can be found in \cite{SReg:finitenesslengthsome}. Furthermore, in \cite{MMV:Highergenerationsubgroup}, they were used to study the BNS-invariants of right-angled Artin groups; higher generation also arises in the context of Deligne complexes (\cite{CD:K1problemhyperplane} and \cite[Example A.7]{BFM+:braidedThompson'sgroups}) and braid groups (\cite{BFM+:braidedThompson'sgroups}).
However, the probably best-known example of higher generating families was given in \cite[Theorem 3.3]{AH:Highergenerationsubgroups}: The set of parabolic subgroups forms a higher generating family for any group with a BN-pair. To show this, Abels and Holz use the theory of Tits buildings. 

The aim of this note is to show that with small adjustments, the result of Abels and Holz can be extended to the general setting of groups acting appropriately on Cohen--Macaulay complexes. Cohen--Macaulayness is a combinatorial property of simplicial complexes defined via local connectivity conditions (see Section \ref{sec CM}). Our main result is Theorem \ref{thm higher generation by CM} which gives a criterion for obtaining higher generating families from group actions on Cohen--Macaulay complexes.  We also give a characterisation of the class of pairs $(G, \mcH)$ which can be obtained that way in Theorem \ref{thm characterisation CC CM}.

As an application of this, we construct new higher generating families: The first one is the family of Levi subgroups in groups with a BN-pair (see Theorem \ref{thm higher generation opposition complex} and Corollary \ref{cor high gen Levi}), the second one is the family of ``parabolic'' subgroups of $\Aut(F_n)$, the automorphism group of the free group (see Definition \ref{def parabolics in Aut(Fn)} and Theorem \ref{cor higher generation Aut}). The corresponding Cohen--Macaulay complexes are the opposition complex and the free factor complex, respectively.

I would like to thank my supervisor Kai-Uwe Bux for his support and Herbert Abels, Stephan Holz and Yuri Santos Rego for many interesting conversations about higher generation and coset complexes. I would also like to thank Russ Woodroofe for several remarks considering the relation between the homological and homotopical versions of Cohen--Macaulayness and for pointing out the reference \cite{Lut:TriangulatedManifoldsFew}. Thanks are also due to an anonymous referee for precise and helpful comments.

The author was supported by the grant BU 1224/2-1 within the Priority Programme 2026 ``Geometry at infinity'' of the German Science Foundation (DFG).

\section{Definitions and general results}
Throughout this text, we will often identify a simplicial complex $X$ and its geometric realisation $\real{X}$ if what is meant is clear from the context.
\subsection{Higher generating subgroups}
\label{sec higher generation}
\begin{definition}
Let $X$ be a set and $\mathcal{U}$ be a collection of subsets of $X$ such that \mcU\, covers $X$. Then the \emph{nerve} $N(\mathcal{U})$ of the cover $\mcU$ is the simplicial complex that has vertex set \mcU\, and where the vertices $U_0,\ldots ,U_k\in \mathcal{U}$ form a simplex if and only if $U_0\cap\ldots \cap U_k\not = \emptyset$.
\end{definition}

\begin{definition}
Let $G$ be a group and \mcH\, a family of subgroups of $G$.
\begin{enumerate}
\item The collection of cosets $\mcU\coloneqq \set[gH \mid g\in G, \, H\in \mcH]$ is a covering of $G$ and we define the \emph{coset complex} $\CC{G}{\mcH}$ to be the nerve $N(\mcU)$. This complex is endowed with a natural action of $G$ given by left multiplication.
\item We say that \mcH\, is \emph{$\TheDim$-generating} for  $G$ if $\CC{G}{\mcH}$ is $(\TheDim-1)$-connected, i.e. $\pi_i\CC{G}{\mcH}=\set[1]$ for all $i<\TheDim$.
\end{enumerate}
\end{definition}

Interesting examples of coset complexes are given by the ``coset poset'' of all subgroups of a finite group as studied in \cite{Bro:cosetposetprobabilistic} and \cite{SW:Ordercomplexescoset}.
The term ``higher generating subgroups'' was coined by Holz in \cite{Hol:EndlicheIdentifizierbarkeitvon} and is motivated by the following: The family $\mcH$ is 1-generating for $G$ if and only if the union of the subgroups in $\mcH$ generates $G$. It is 2-generating if and only if $G$ is the free product of the subgroups in $\mcH$ amalgamated along their intersections. Roughly speaking, the latter means that the union of the subgroups generates $G$ and that all relations that hold in $G$ follow from relations in these subgroups. $3$-generation can similarly be defined using identities among relations (see \cite[2.8]{AH:Highergenerationsubgroups}).

\begin{remark}
\label{remark k simplices}
The cosets $g_0H_0,\ldots, g_\TheCodim H_\TheCodim$ with $ g_i\in G$ and $H_i\in \mcH$ intersect non-trivially if and only if there is $g\in G$ such that 
\begin{equation*}
g_0H_0\cap\ldots\cap g_\TheCodim H_\TheCodim=g(H_0\cap\ldots\cap H_\TheCodim).
\end{equation*}
Hence, the set of $k$-simplices of $\CC{G}{\mcH}$ is in bijection with the set
\begin{align*}
\set[g(H_0\cap\ldots\cap H_\TheCodim) \mid   g\in G, \, H_i\in \mcH, \, H_i\not=H_j \text{ for } i\not=j].
\end{align*}
\end{remark}

Assume that $\mcH$ is a finite family of subgroups of $G$. Then $\CC{G}{\mcH}$ has dimension $|\mcH|-1$ and $\mcH$ itself is the vertex set of a  \emph{facet}, i.e. a maximal simplex, of the coset complex. We will write this facet as $C_\mcH$. This (and hence any other) facet is a \emph{fundamental domain} for the action of $G$; this means that
for all $0\leq k\leq |\mcH|-1$, the set of $k$-faces of $C_\mcH$ contains exactly one element of each $G$-orbit of $k$-simplices of
$\CC{G}{\mcH}$. The following converse of this observation is due to Zaremsky.

\begin{lemma}[see {\cite[Proposition A.5]{BFM+:braidedThompson'sgroups}}]
\label{lem detecting CC}
Let $G$ be a group acting by simplicial automorphisms on a simplicial complex $X$, with a single facet $C$ as fundamental domain.
Let
\begin{equation*}
\mathcal{P}\coloneqq \set[\Stab_G(v) \mid v \text{ is a vertex of } C].
\end{equation*}
Then the map
\begin{align*}
\psi:\CC{G}{\mathcal{\mcP}}&\to X\\
 g\Stab_G(v)&\mapsto g.v
\end{align*}
is an isomorphism of simplicial $G$-complexes.
\end{lemma}

\subsection{The Cohen--Macaulay property}
For the remainder of this section, let $\mathbf{k}$ be a field or the ring of integers $\mathbb{Z}$.

\label{sec CM}
\begin{definition}
Let $X$ be a simplicial complex of dimension $d<\infty$. Then $X$ is \emph{Cohen--Macaulay over $\mathbf{k}$} if it is $(d-1)$-acyclic over $\mathbf{k}$, i.e. $\tilde{H}_i(X, \mathbf{k})=\set[0]$ for all $i<d$, and the link of every $s$-simplex is $(d-s-2)$-acyclic over $\mathbf{k}$.

$X$ is \emph{homotopy Cohen--Macaulay} if it is $(d-1)$-connected and the link of every $s$-simplex is $(d-s-2)$-connected.
\end{definition}

The notion of Cohen--Macaulayness over $\mathbf{k}$ was introduced in the mid-70s and came up in the study of finite simplicial complexes via their Stanley-Reisner rings (see \cite{Sta:CombinatoricsCommutativeAlgebra:}). The homotopical version was introduced by Quillen in \cite{Qui:Homotopypropertiesposet}. 
While it can be shown that ``being Cohen--Macaulay over $\mathbf{k}$'' only depends on the geometric realisation $\real{X}$ and not on its specific triangulation, the homotopical version is not a topological invariant but a property of the simplicial complex $X$ itself.
One has implications
\begin{align*}
&\text{homotopy CM} & \Rightarrow && \text{CM over }\mathbb{Z} &&\Rightarrow &&\text{CM over any field } \mathbf{k},
\end{align*}
which are all strict.
For more details on Cohen--Macaulayness and its connections to other combinatorial properties of simplicial complexes, see \cite{Bjo:Topologicalmethods}. 
We will talk about examples of complexes having these properties in Section \ref{sec Higher generation and CM} and Section \ref{sec:applications}.

\begin{definition}
A topological space is \emph{$d$-spherical} if it is homotopy equivalent to a wedge of $d$-spheres; as a convention, we consider a singleton to be homotopy equivalent to a (trivial) wedge of $n$-spheres for all $n$.
\end{definition}
\begin{remark}
By the Whitehead theorem, a $d$-dimensional complex is $d$-spherical if and only if it is $(d-1)$-connected.
\end{remark}

An advantage of a complex that is Cohen--Macaulay over one that is merely spherical is that it allows for inductive methods using its local structure. This is also what we will make use of in the proof of the following lemma.

\begin{lemma}
\label{lem spherical complement}
Let $X$ be a $d$-dimensional complex and let $X_s:=\real{X}\bsl \lVert X^{(s)} \rVert $ denote the complement of the $s$-skeleton of $\real{X}$. The following holds true:
\begin{enumerate}
\item If $X$ is Cohen--Macaulay over $\mathbf{k}$, the homology with $\mathbf{k}$-coefficients of $X_s$ is concentrated in dimension $d-s-1$, i.e. $\tilde{H}_i(X_s, \mathbf{k})$ is trivial if $i\not=d-s-1$.
\item If $X$ is homotopy Cohen--Macaulay, $X_s$ is $(d-s-1)$-spherical.
\end{enumerate}
\end{lemma}
\begin{proof}
The proof of the two statements is completely parallel and will be done by induction on $s$. Setting $X_{-1}\coloneqq\real{X}$, the statements hold  for $s=-1$ as $\real{X}$ itself is assumed to be $(d-1)$-acyclic or $(d-1)$-connected, respectively. For all $s$, the space $X_{s-1}$ is the union of $X_{s}$ and the open $s$-simplices of $\real{X}$, so we will successively adjoin  these simplices to $X_s$ while keeping track of the homotopy type. Assume that we already constructed $X'$ as the union of $X_s$ and a set of open $s$-simplices of $\real{X}$. Then for every $s$-simplex $\sigma$ in $\real{X}$ that is not contained in $X'$, 
there is an open contractible neighbourhood $U$ of the interior of $\sigma$ in $X''\coloneqq X'\cup  \mathring{\sigma}$ such that $U\cap X'=U\bsl \mathring{\sigma}$ is homotopy equivalent to the link of $\sigma$ in $X$.
As $X$ is Cohen--Macaulay, this link is $(d-s-2)$-acyclic in the homological and $(d-s-2)$-connected in the homotopical setting.
This means that $X''$ can be constructed gluing together $X'$ 
and $U$, which is contractible, along the open subset $U\bsl \mathring{\sigma}$, which is $(d-s-2)$-acyclic or $(d-s-2)$-connected.  
Hence, the inclusion $X'\hookrightarrow X''$ induces for all $i\leq d-s-2$ an isomorphism on homology groups $\tilde{H}_i(\cdot, \mathbf{k)}$ or homotopy groups $\pi_i$, respectively.

By induction, we can conclude that if $X$ is Cohen--Macaulay over $\mathbf{k}$, we have $\set[0]=\tilde{H}_i(X_{s-1},\mathbf{k})\cong \tilde{H}_i(X_s,\mathbf{k})$ and if it is homotopy Cohen--Macaulay, we have $\set[1]=\pi_i(X_{s-1})\cong \pi_i(X_s)$ for $i\leq d-s-2$. Noting that the complement of the $s$-skeleton of any simplicial complex of dimension $d$ is homotopy equivalent to a complex of dimension $(d-s-1)$ (contract all the simplices of dimension $(s+1)$ to their barycenters), the result follows.
\end{proof}

A simplicial complex is called \emph{pure} if all of its facets have the same dimension. A pure simplicial complex $X$ is called a \emph{chamber complex} (or \emph{strongly connected}) if every pair of facets $\sigma,\tau\in X$ can be connected by a sequence of facets $\sigma=\nolinebreak\tau_1,\ldots, \tau_k=\tau$ such that for all $1\leq i \leq k$, the intersection of $\tau_i$ and $\tau_{i+1}$ is a face of codimension $1$. The facets of a chamber complex are also called \emph{chambers}.

\begin{remark}
Every Cohen--Macaulay complex is pure and a chamber complex (see e.g. \cite[Proposition 11.7]{Bjo:Topologicalmethods}). The preceding lemma is a generalisation of this well-known fact in the following sense: Let $X$ be a pure, $d$-dimensional simplicial complex, $d\geq 1$. Define a graph $\Gamma$ whose vertices  are given by the facets of $X$ and where two vertices are joined by an edge if and only if the corresponding facets intersect in a face of codimension $1$. The graph $\Gamma$, which is also called the \emph{chamber graph of $X$}, is homotopy equivalent to the complement of the $(d-2)$-skeleton of $X$. Furthermore, $X$ is a chamber complex if and only if $\Gamma$ is connected, which is equivalent to $\tilde{H}_0(\Gamma)= \set[0]$. So if we assume that $X$ is Cohen--Macaulay, Lemma \ref{lem spherical complement} implies that it is a chamber complex.
\end{remark}

\subsection{Higher generation by actions on Cohen--Macaulay complexes}

\label{sec Higher generation and CM}
We now want to combine Lemma \ref{lem detecting CC} with the observations of the preceding subsection in order to obtain higher generating families of subgroups for groups acting on Cohen--Macaulay complexes. Our main result here is as follows.

\begin{theorem}
\label{thm higher generation by CM}
Let $G$ be a group acting by simplicial automorphisms on a simplicial complex $X$, with a single facet $C$ as fundamental domain. If $X$ is homotopy Cohen--Macaulay and has dimension $d$, the set
\begin{equation*}
\mathcal{P}_\TheCodim\coloneqq \set[\Stab_G(\sigma)\mid \sigma \text{ is a $\TheCodim$-dimensional face of } C]
\end{equation*}
is $(d-\TheCodim)$-generating for all $0\leq \TheCodim\leq d$. Furthermore, the corresponding coset complex $\CC{G}{\mcP_k}$ is $(d-k)$-spherical.
\end{theorem}
\begin{proof}
By Lemma \ref{lem detecting CC}, we can identify $X$ with the coset complex \CC{G}{\mcP_0}.

As $C$ is a fundamental domain for the action of $G$, the stabiliser of a $\TheCodim$-face $F$ of $C$ is equal to the intersection of the stabilisers of all the vertices of $F$. Hence, the elements of $\mcP_\TheCodim$ 
are given by all the intersections of $(\TheCodim+1)$ pairwise distinct elements from $\mcP_0$.

By Remark \ref{remark k simplices}, the vertices of $\CC{G}{\mcP_\TheCodim}$ are in one-to-one correspondence with the $k$-simplices of $\CC{G}{\mcP_0}\cong X$. Moreover, a set of vertices in  $\CC{G}{\mcP_\TheCodim}$ forms a simplex if and only if the corresponding $k$-simplices in $X$ are all faces of one common facet.
It follows that the geometric realisation $\real{\CC{G}{\mcP_\TheCodim}}$ is homotopy equivalent to $\real{Y}$, where $Y$ is the induced subcomplex of the barycentric subdivision $\mathcal{B}(X)$ whose vertices are the barycenters of all simplices of  $X$ that have dimension greater or equal to $k$. 

The complex $\real{Y}$ is homotopy equivalent to the complement of the $(k-1)$-skeleton of $\real{X}$. As $X$ is Cohen--Macaulay, we can use Lemma \ref{lem spherical complement} to conclude that $\CC{G}{\mcP_\TheCodim}$ is $(d-\TheCodim)$-spherical. This finishes the proof.
\end{proof}

Before we apply this theorem to obtain higher generating families of subgroups for specific examples in the next section, we now characterise the class of pairs $(G, \mcH)$ which can be obtained using Theorem \ref{thm higher generation by CM}: By Lemma \ref{lem detecting CC}, the conditions of Theorem \ref{thm higher generation by CM} are fulfilled if and only if $\CC{G}{\mcP_0}$ is homotopy Cohen--Macaulay. We will give an alternative characterisation of this condition for coset complexes.

A pure simplicial complex $X$ of dimension $d$ is called \emph{coloured} (or \emph{completely balanced}) if there is a map $c:X^{(0)}\to \set[0,\ldots, d]$ restricting to a bijection on each facet. In this setting, for each $J\subseteq \set[0,\ldots,d]$, let $X_J$ be the induced subcomplex of $X$ with vertex set $c^{-1}(J)$. As stated below, the following result is due to Walker:

\begin{theorem}[{\cite[Theorem 5.2]{BWW:sequentiallyCohenMacaulay}}, {\cite[Theorem 11.14]{Bjo:Topologicalmethods}}]
\label{thm characterisation pure CM}
Let $X$ be a pure $d$-dimensional coloured complex. Then $X$ is Cohen--Macaulay over $\mathbf{k}$ if and only if $X_J$ is $(|J|-2)$-acyclic over $\mathbf{k}$ for every $J\subseteq \set[0,\ldots, d]$. It is homotopy Cohen--Macaulay if and only if $X_J$ is $(|J|-2)$-connected for every $J\subseteq \set[0,\ldots, d]$.
\end{theorem}

Every finite-dimensional coset complex is a pure simplicial complex which can be given a colouring 
\begin{equation*}
c:\CC{G}{\set[H_0, \ldots, H_d]}\to \set[0,\ldots, d]
\end{equation*}
by setting $c(g H_i)\coloneqq i$. 
Hence, the following is an immediate consequence of Theorem \ref{thm characterisation pure CM}:

\begin{theorem}
\label{thm characterisation CC CM}
Let $G$ be a group and \mcH\, be a finite family of subgroups of $G$. 
\begin{enumerate}
\item $\CC{G}{\mcH}$ is Cohen--Macaulay over $\mathbf{k}$ if and only if for all $\mcH'\subseteq \mcH$, the coset complex $\CC{G}{\mcH'}$ is $(|\mcH'|-2)$-acyclic over $\mathbf{k}$.
\item $\CC{G}{\mcH}$ is homotopy Cohen--Macaulay if and only if every $\mcH'\subseteq \mcH$ is $(|\mcH'|-1)$-generating for $G$.
\end{enumerate}
\end{theorem}

Being a coset complex imposes rather strong restrictions: In addition to being coloured, every such complex is endowed with a facet transitive group action. One might ask whether in this setting, Cohen--Macaulayness implies already stronger combinatorial conditions like shellability. A finite complex is \emph{shellable} if and only if the set of its facets admits a sufficiently nice ordering, called a \emph{shelling}; for the precise definition, see \cite[Section 11.2]{Bjo:Topologicalmethods}. In general, being shellable is strictly stronger than being homotopy Cohen--Macaulay. Buildings form a class of coset complexes which are shellable (see Section \ref{sec:applications} and\cite{Bjo:Somecombinatorialalgebraic}). The following example however shows that there are also coset complexes which are Cohen--Macaulay over $\mathbb{Z}$, but are not homotopy Cohen--Macaulay and so in particular not shellable.

Let $\operatorname{Alt}_5$ be the alternating group on the set $\set[1,2,3,4,5]$ and consider the following subgroups:
\begin{align*}
H_1&\coloneqq \Stab_{\operatorname{Alt}_5}(\set[2]), \\
H_2&\coloneqq N_{\operatorname{Alt}_5}(\,\langle (1,2,3,4,5) \rangle\,), \\
H_3&\coloneqq N_{\operatorname{Alt}_5}(\,\langle (1,3,5) \rangle\,),
\end{align*}
where $\Stab_{\operatorname{Alt}_5}$ and $N_{\operatorname{Alt}_5}$ denote stabiliser and normaliser in $\operatorname{Alt}_5$. The group $H_1$ is isomorphic to $\operatorname{Alt}_4$ and $H_2$ and $H_3$ are isomorphic to the dihedral groups $D_5$ and $D_3$, respectively.
Let $\mcH\nolinebreak\coloneqq\set[H_1,H_2,H_3]$. The coset complex $\CC{\operatorname{Alt}_5}{\mcH}$ has dimension two and consists of $21$ vertices, $80$ edges and $60$ two-simplices.
This complex was first found by Oliver, an explicit description of it as  a coset complex can be found in \cite{Seg:Groupactionsfinite}. For further details and a picture, see \cite[Section 7.3]{Lut:TriangulatedManifoldsFew}; note that $\CC{\operatorname{Alt}_5}{\mcH}$ is isomorphic to the complex $N_0$ in \cite{Lut:TriangulatedManifoldsFew}.

\begin{lemma}
The coset complex $\CC{\operatorname{Alt}_5}{\mcH}$ is Cohen--Macaulay over $\mathbb{Z}$, but is not homotopy Cohen--Macaulay.
\end{lemma}
\begin{proof}
In \cite{Lut:TriangulatedManifoldsFew}, Lutz shows that $\real{\CC{\operatorname{Alt}_5}{\mcH}}$ is homeomorphic to a cell complex $Q$ obtained by taking the boundary of a dodecahedron and identifying opposite pentagons by a coherent twist of $\pi/5$. The complex $Q$ arises in triangulations of the Poincar\'e homology 3-sphere $\Sigma^3$. It is $\mathbb{Z}$-acyclic and one has $\pi_1(Q)\cong \pi_1(\Sigma^3)$ (see \cite[p. 57]{Bre:Introductioncompacttransformation}).
As this fundamental group is non-trivial, $Q$ and therefore $\CC{\operatorname{Alt}_5}{\mcH}$ cannot be homotopy Cohen--Macaulay.

It remains to show that $\CC{\operatorname{Alt}_5}{\mcH}$ is Cohen--Macaulay over $\mathbb{Z}$. By Theorem \ref{thm characterisation CC CM}, it suffices to show that for all $\mcH'\subseteq \mcH$, the complex $\CC{\operatorname{Alt}_5}{\mcH'}$ is $(|\mcH'|-2)$-acyclic. For $\mcH'=\mcH$, this is true as $Q$ is $\mathbb{Z}$-acyclic and for $|\mcH'|=1$, there is nothing to show. Hence, one only needs to check that for all two-element subsets $\mcH'$ of $\mcH$,
the corresponding subcomplex of $\CC{\operatorname{Alt}_5}{\mcH}$ is connected. This can easily be verified, e.g. by using Figure 7.5 of \cite{Lut:TriangulatedManifoldsFew}.
\end{proof}

A further question in the same direction which might be interesting to consider is whether every coset complex that is \emph{homotopy} Cohen--Macaulay is already shellable. A counterexample to that (if existent) would have to be a pure, completely balanced simplicial complex with a facet-transitive group action that is homotopy Cohen--Macaulay but not shellable. To us, it seems likely that such a complex exists but we are currently not aware of any examples.

\section{Applications}
\label{sec:applications}
In what follows, we give three applications of Theorem \ref{thm higher generation by CM}. All of them are either directly or in spirit connected to the theory of buildings.
The first two examples make direct use of this theory; definitions and background material needed for these subsections can be found in \cite{AB:Buildings}. The third application concerns higher generating families of subgroups in $\Aut(F_n)$, the automorphism group of the free group.

\subsection{Parabolic subgroups and buildings}
Our first application recovers \cite[Theorem 3.3]{AH:Highergenerationsubgroups} of Abels and Holz. We will be brief here and refer to their text for further details.

Let $G$ be a group with a BN-pair, denote by $\Delta$ the corresponding building and by $Ch(\Delta)$ the set of its chambers. If the corresponding Weyl group $W$ has rank $r$, this building is homotopy equivalent to a wedge of $(r-1)$-spheres by the Solomon--Tits theorem (see \cite{Sol:Steinbergcharacterfinite}); it is in fact contractible if $W$ is infinite. The link of a simplex of dimension $k$ in $\Delta$ is again a building of rank $r-k-1$ which implies that $\Delta$ is homotopy Cohen--Macaulay. 

The action of $G$ is transitive on the chambers of $\Delta$, so we can apply Theorem \ref{thm higher generation by CM} to deduce that for any choice of chamber $C\in Ch(\Delta)$, the family $\mcP_k$ of stabilisers of the $k$-dimensional faces of $C$ is $(r-1-k)$-generating for $G$. If we take $C$ to be the ``fundamental'' chamber associated to the Borel subgroup $B$, these stabilisers are exactly the \emph{standard parabolic subgroups of rank $r-k-1$}. Hence we get:
\begin{theorem}[{\cite[Theorem 3.3]{AH:Highergenerationsubgroups}}]
The family of rank-$m$ standard parabolic subgroups is $m$-generating for $G$. 
\end{theorem}

\subsection{Levi subgroups and the opposition complex}
To show that the families of standard parabolic subgroups in a group $G$ with a BN-pair are higher generating, we only needed to use chamber-transitivity of the action of $G$ on the associated building. However, this action is known to satisfy stronger transitivity conditions; we will exploit them to find other families of higher generating subgroups in this subsection.

Let $\Delta$ be a spherical building. The chamber distance $d(-,-)$ induces an \emph{opposition relation} $\operatorname{op}$ between chambers of $\Delta$ which is defined by
\begin{equation*}
C\operatorname{op} C' :\Leftrightarrow d(C,C')=\max\set[d(C_1,C_2)\mid C_1,C_2\in Ch(\Delta)].
\end{equation*}
This opposition relation can be extended to arbitrary simplices $\sigma,\sigma'\in\Delta$ of equal dimension by saying that $\sigma$ is opposite to $\sigma'$ if and only if the following holds true:

\begin{center}
\begin{minipage}{0.85\textwidth}
\emph{For every chamber $C\geq \sigma$ in $\Delta$, there is a chamber $C'\geq\sigma'$ such that $C\op C'$ and for every chamber $C'\geq \sigma'$, there is a chamber $C\geq\sigma$ such that $C\op C'$.}
\end{minipage}
\end{center}

\noindent Using this opposition relation, one can define a new complex from $\Delta$ as follows:
\begin{definition}
The \emph{opposition complex} \Opp\, is the simplicial complex whose simplices are of the form $(\sigma,\sigma')$ with $\sigma,\sigma'\in\Delta$, $\sigma\op\sigma'$
and where the face relation is given by
\begin{equation*}
(\tau,\tau')\leq (\sigma,\sigma'):\Leftrightarrow\tau\leq\sigma\text{ and } \tau'\leq\sigma'.
\end{equation*}
\end{definition}

\Opp\, has the same dimension as $\Delta$ and it was shown to be homotopy Cohen--Macaulay by von Heydebreck in \cite{vHey:Homotopypropertiescertain}. The complex is pure and its facets are given by pairs $(C,C')$ of opposite chambers $C,C'\in Ch(\Delta)$.

Every building $\Delta$ comes with a map
\begin{align*}
\delta : Ch(\Delta) \times Ch(\Delta)\to W,
\end{align*}
where $W$ is the Weyl group of $\Delta$. This function is called the \emph{Weyl distance function} (of $\Delta$) and it is related to the gallery distance as follows:
\begin{align*}
d(C,C')= l_S(\delta(C,C')),
\end{align*}
where $l_S$ denotes the Coxeter length function on $W$. If a group acts by type-preserving automorphisms on $\Delta$, we say that the action is \emph{Weyl transitive} if for each $w\in W$, the action is transitive on the set of order pairs of chambers $(C,C')$ with $\delta(C,C')=w$.

\begin{theorem}
\label{thm higher generation opposition complex}
Let $G$ be a group acting Weyl transitively by type-preserving automorphisms on a spherical building $\Delta$ of dimension $d$. Choose any pair $(C,C')$ of opposite chambers $C,C'\in Ch(\Delta)$. Then the set
\begin{equation*}
\mcP_\TheCodim\coloneqq\set[\Stab_G(\sigma)\cap \Stab_G(\sigma') \mid \sigma,\sigma' \text{ $\TheCodim$-dimensional faces of } C,C';\, \sigma\op\sigma']
\end{equation*}
is $(d-\TheCodim)$-generating for $G$.
\end{theorem}
\begin{proof}
As the action of $G$ on $\Delta$ preserves distances and adjacency relations, it induces a simplicial action on \Opp\, given by
\begin{equation*}
g.(\sigma,\sigma')\coloneqq(g.\sigma,g.\sigma').
\end{equation*} 
We claim that the simplex $(C,C')\in\Opp$ is a fundamental domain for this action of $G$. Because $\Delta$ is spherical, its Weyl
group $W$ is finite and has a unique element $w_S$ of maximal length. Hence, two chambers $D,D'\in Ch(\Delta)$ are opposite to each other if and only if $\delta(D,D')=w_S$ and by Weyl transitivity, $G$ acts transitively on such pairs of opposite chambers. This implies that the set of vertices of $(C,C')$ contains a representative of each $G$-orbit of vertices in $\Opp$. Furthermore, the type of any vertex of the chamber $C$ is preserved by all the elements of $G$. Hence, no two distinct vertices of $(C,C')$ lie in the same $G$-orbit which proves that this facet is indeed a fundamental domain.

As a consequence, Theorem \ref{thm higher generation by CM} shows that the set $\mcP_k$ of stabiliser of $k$-simplices in $\Opp$ is $(d-k)$-generating. Since a $k$-simplex in $\Opp$ is a pair $(\sigma,\sigma')$ of $k$-simplices $\sigma,\sigma'\in \Delta$, this finishes the proof.
\end{proof}

In particular, the conditions of the preceding theorem are fulfilled in the following situation: If $G$ is a group having a BN-pair of rank $r$ with finite Weyl group $W=\nolinebreak\ll S \rr$, it acts Weyl transitively on the associated spherical building. The chambers associated to $B$ and $B^-=w_S B w_S$ are opposite to each other and after setting $C\coloneqq B$ and $C' \coloneqq B^-$, the family $\mcP_k$ defined in Theorem \ref{thm higher generation opposition complex} is the set of \emph{standard rank-$(r-k-1)$ Levi subgroups}. We state this as follows:

\begin{corollary}
\label{cor high gen Levi}
Let $(G,B,N,S)$ be a Tits system with finite Weyl group. Then the family of standard rank-$m$ Levi subgroups is $m$-generating for $G$.
\end{corollary}

\begin{example}
As an illustration, we spell out the following special case of this result:
If $\Delta$ is the flag complex of proper subspaces of the vector space $\mathbf{k}^{\TheRank}$, i.e. a building of type $A_{n-1}$, the opposition complex \Opp\, is the complex with vertex set
\begin{equation*}
\set[(U,U') \mid U,\,U'\text{ are proper subspaces of }\mathbf{k}^{\TheRank}\text{ and } U\oplus U'=\mathbf{k}^{\TheRank}]
\end{equation*} 
in which $(U_0,U_0'),\ldots ,(U_k,U_k')$ form a simplex if and only if 
 (possibly after reordering), one has $U_0<U_1<\ldots <U_k$ and $U'_0>U'_1>\ldots >U'_k$.

Let $\set[e_1,\ldots, e_n]$ be the standard basis of $\mathbf{k}^n$. The flags 
\begin{align*}
&C\coloneqq \ll e_1 \rr < \ll e_1, e_2 \rr < \ldots < \ll e_1,\ldots, e_{n-1} \rr \text{ and }\\
& C'\coloneqq \ll e_2,\ldots, e_{n} \rr > \ll e_3,\ldots, e_{n} \rr > \ldots > \ll e_n \rr
\end{align*}
form opposite chambers of $\Delta$. 
The building $\Delta$ has dimension $n-2$ and $\GL{\mathbf{k}}$ acts Weyl transitively on it. The corresponding family of stabilisers $\mcP_k$ with $0\leq k\leq n-3$ consists of all subgroups of the form
\begin{equation*}
\begin{pmatrix}
\GL[n_1]{\mathbf{k}} 	&0	&\cdots	&0\\
0	&\GL[n_2]{\mathbf{k}}	&	&\vdots\\
\vdots	&	&\ddots	&0\\
0	&	\cdots&	0&\GL[n_{k+2}]{\mathbf{k}}\\
\end{pmatrix}\leq \GL{\mathbf{k}}.
\end{equation*}
So the number of blocks in the corresponding matrices is $k+2$ and the $n_i$ are natural numbers 
such that $\sum_{i=1}^{k+2} n_i=n$.
These are exactly the standard rank-$(\TheRank-2-k)$ Levi subgroups of $\GL{\mathbf{k}}$ and by Theorem \ref{thm higher generation opposition complex}, this family is $(n-2-k)$-generating.

\end{example}

\subsection{Parabolics in Aut($F_n$) and the free factor complex}
Hatcher and Vogtmann in \cite{HV:complexfreefactors} defined a simplicial complex associated to $\Aut(F_n)$, the automorphism group of the free group on $n$ letters. It is similarly defined as and shares many properties with the building associated to $\GL[n]{\mathbb{Z}}=\Aut(\mathbb{Z}^n)$.

\begin{definition}\label{Def FC}
A subgroup $A$ of $F_n$ is called a \emph{free factor}, if there is a subgroup $B\leq F_n$ such that $F_n$ can be written as a free product $F_n=A\ast B$.

The \emph{free factor complex} \FC\, is the simplicial complex whose vertices are proper free factors of $F_n$ and where the vertices $H_0,\ldots ,H_k$ form a simplex if and only if they form a flag $H_0\leq H_1\leq \ldots \leq H_k$.
\end{definition}

\FC\, is a chamber complex of dimension $\TheRank -2$ and comes with a simplicial action of $\Aut(F_n)$ given by $g.(H_0<H_1<\ldots <H_k)\coloneqq g(H_0)<g(H_1)<\ldots <g(H_k)$. 
A fundamental domain for this action is given by any maximal flag $H_0<\ldots <H_{\TheRank-2}$ of free factors in $F_\TheRank$.

\begin{definition}
\label{def parabolics in Aut(Fn)}
Fix a basis $b_1,\ldots,b_n$ of $F_n$. This gives rise to a ``standard-flag''
\begin{equation*}
C\coloneqq \ll b_1 \rr<\ll b_1,b_2 \rr<\ldots < \ll b_1,\ldots,b_{n-1}\rr
\end{equation*}
of free factors in $F_\TheRank$.
Now analogously to the situation in buildings, we define a \emph{standard rank-\TheDim\, parabolic subgroup} 
to be the stabiliser of a sub-flag of $C$ that has length $\TheRank-\TheDim-1$.
To match the numbering of Section \ref{sec Higher generation and CM}, we use the corank to number the parabolic subgroups and define $\mcP_{\TheRank-\TheDim-2}$ to be the set of standard rank-$m$ parabolics.
\end{definition}

Again, we use Theorem \ref{thm higher generation by CM} to show:

\begin{theorem}
\label{cor higher generation Aut}
For all $1\leq\TheDim\leq\TheRank-2$, the family $\mcP_{\TheRank-\TheDim-2}$ of standard rank-$\TheDim$ parabolic subgroups is $\TheDim$-generating for $\Aut(F_n)$. The corresponding coset complex $\CC{\Aut(F_n)}{\mcP_{\TheRank-\TheDim-2}}$ is $\TheDim$-spherical.
\end{theorem}
\begin{proof}
As noted above, $\Aut(F_n)$ acts on the the free factor complex with the facet $C=\ll b_1 \rr<\ll b_1,b_2 \rr<\ldots < \ll b_1,\ldots,b_{n-1}\rr$ as a fundamental domain.  The standard rank-$\TheDim$ parabolics are exactly the stabilisers of the $(\TheRank-\TheDim-2)$-faces of $C$. Since Hatcher and Vogtmann showed that $\FC$ is homotopy Cohen--Macaulay (see \cite[Section 4]{HV:complexfreefactors}), the statement is an immediate consequence of Theorem \ref{thm higher generation by CM}.
\end{proof}

\bibliographystyle{halpha}
\bibliography{bibliography}

\end{document}